\documentclass[11pt]{amsart}
\usepackage{amsmath}%
\usepackage{amsfonts}%
\usepackage{amssymb}%
\usepackage{graphicx}

\theoremstyle{plain}
\newtheorem{theorem}{Theorem}[section]
\newtheorem{lemma}[theorem]{Lemma}

\newtheorem{proposition}[theorem]{Proposition}

\theoremstyle{definition}

\theoremstyle{remark}

\newcommand{\Hb}{{\bf H}}

\newcommand{\CC}{{\mathbb C}}

\begin{document}
\title{On  special matrices related to Cauchy and Toeplitz  matrices}
\author{Sajad Salami}
\address[Sajad Salami]{Inst\'{i}tuto da Matem\'{a}tica e Estat\'{i}stica \\
Universidade Estadual do Rio do Janeiro, Brazil}%
\email[Sajad Salami]{sajad.salami@ime.uerj.br}%
\urladdr{https://sites.google.com/a/ime.uerj.br/sajadsalami/}
\thanks{}
\date{}
\begin{abstract}
In this paper,  we are going to calculate the determinant of a certain type of square matrices, which are related to the well-known Cauchy and Toeplitz matrices. Then, we will use the results to determine the rank of special non square matrices.
\end{abstract}
\maketitle

\section{Introduction and main result}
\label{int}
Throughout this paper, we fix a field $F$  of characteristic zero, $2 \leq r < n$ arbitrary  integers, and  $\{a_\ell\}_{\ell=1}^\infty \subset F^*$  an infinite sequence of distinct elements. 
For  any pair of indexes $(\ell,e) $, we define $d_{(\ell,e)}:=a_\ell-a_e$ and $d_\ell:=d_{(\ell+1,\ell)}$. 
Then, we consider the $(n-r)\times(r+1)$ matrix $C:=[C_{(i-r,j)}]$, 
such that for  $j=0,\cdots, r$ and $r+1\leq i \leq n$ we have
$$ C_{(i-r,0)} 
=(-1)^{r}  \prod_{ e< \ell  \in I}  d_{(i,\ell)} d_{(\ell,e)},\ 
C_{(i-r,j)}  
=(-1)^{r+j} a_i  \prod_{ e< \ell  \in I_j} a_\ell d_{(i,\ell)} d_{(\ell,e)},
$$
$$D_r = (-1)^{r}  \prod_{ e< \ell  \in I}a_\ell  (a_\ell -a_e)=(-1)^{r}   \prod_{ e< \ell  \in I} a_\ell d_{(\ell, e)}\not =0,$$
where $I=\{ 1,\cdots ,r\}$ and $I_j:= I \backslash \{j\}$ for  $ j\in I$.  
Define ${\bf C}_r^n:=[C|D]$ as  a $(n-r)\times(n+1)$   blocked matrix, where $D$ is a $(n-r)\times(n-r)$ 
diagonal matrix with entries $D_r$. 
A special case of ${\bf C}_r^n$ is related to the Hilbert and   Toeplitz's matrices \cite{hilb1, Choi,  Chu-li}. For more details, see Section \ref{Cau-hilb}.

The main result of this paper concerns with  calculating the rank of the matrix ${\bf C}_r^n$,

\begin{theorem}
	\label{matmain}
	Let $F$ be  a field of characteristic zero and  $\{a_\ell\}_{\ell=1}^\infty \subset F^*$ be an infinite sequence of distinct elements of $F$. Then the  matrix ${\bf C}_r^n$ has  full rank $n-r$ for integers $2\leq r < n$.
\end{theorem}
In order to prove the above theorem,   we will calculate the determinant of certain square matrices which are related to the well-known Cauchy's matrices \cite{cauchy, pol-sz}. 
We notice that it is used the author's  forthcoming paper \cite{salami2} to show the non-singularity of  a certain  family of complete intersection  varieties satisfying  the Bombieri-Lang conjecture  in the Diophantine geometry \cite{lang2}.

The organization of the present paper is as follows.
In the  section 2, we recall the definition and determinant of the square Cauchy's and Hilbert's matrix over a field of characteristic zero. In section 3, we  calculate the determinant of a certain square matrix that are related to the Cauchy's matrix.  Finally, in  Section 4, we use the result of Section 3 to prove Theorem \ref{matmain}.
\section{Cauchy's and Toeplitz matrices}
\label{Cau-hilb}
In 1841, Augustin Louis Cauchy introduced a certain type of matrices with certain  properties, see \cite{cauchy, pol-sz}. 
We are going to recall the definition and determinant of these matrices in this section.

An  $n \times n$ square {\it Cauchy's matrix} defined by disjoint subsets of distinct nonzero elements $\{x_1, \cdots, x_n\}$ and
$ \{y_1, \cdots, y_n\}$  in 
a field of characteristic zero $F$, is the square matrix $X_n :=[x_{ij}] $ with 
$$x_{ij}= \frac{1}{x_i-y_j}, \ 1\leq i, j \leq n.$$

Note that any submatrix of a  Cauchy's matrix  is itself a  Cauchy's matrix.  
The determinant of a Cauchy's matrix is known as {\it Cauchy's determinant} in the literature, which is always nonzero because $x_i \not = y_j$. Following proposition shows that how one can calculate the determinant of  Cauchy's matrices.
\begin{proposition}
	\label{Cauchy}
	Let $n \geq 1$ be an integer and $X_n$   a $n \times n$ Cauchy's matrix  defined as above over a field  $F$ of characteristic zero.
	Then  
	$$|X_n|=\frac{\prod_{i<j\in I} (x_i- x_j) (y_i- y_j)}{\prod_{i \in I} \prod_{j \in I} (x_i- y_j)}, \ I=\{1,2, \cdots, n\}.$$
\end{proposition}
\begin{proof}
	See the lemma (11.3) in  \cite{davis} for an analytic proof, when $F=\CC$. For an arbitrary field $F$, we will use the elementary column and row operations to get the desired result.
	Subtracting the first column of $U$ from others gives that 
	$$x_{ij}=\frac{(y_1- y_j)}{(x_i- y_1)}\cdot \frac{1}{(x_i- y_j)}\ (1 \leq i, j\leq n).$$
	Extracting the factor $1/(x_i- y_1)$ from $i$-th row  for $i=1, \cdots, n$, and $y_1- y_j$ from $j$-th column for $j=2, \cdots, n$ leads to 
	$$|U|= \frac{1}{(x_1- y_1)} \cdot \frac{\prod_{j=2}^n (y_1-y_j)}{\prod_{1=2}^n (x_i-y_1)} \cdot 
	\begin{vmatrix}
	1 & \frac{1}{(x_1-y_2)} & \cdots & \frac{1}{(x_1-y_n)} \\
	1 & \frac{1}{(x_2-y_2)} & \cdots & \frac{1}{(x_2-y_n)} \\
	\vdots  &  \cdots  &  \vdots  & \vdots\\
	1 & \frac{1}{(x_n-y_2)} &  \cdots  &  \frac{1}{(x_n-y_n)}     
	\end{vmatrix}.$$
	Now, denoting the last determinant by $|x'_{ij}|$ and   subtracting its first row  from others,  we get
	$$x'_{i1}=0, \  x'_{ij}=\frac{(x_1- x_i)}{(x_1-y_j)}\cdot \frac{1}{(x_i-y_j)}\ 2 \leq i, j\leq n.$$
	Extracting the factor $(x_1- x_i)$ from each rows, and $1/(x_1-y_j)$ from each column, for $2 \leq i,j  \leq n$, gives that
	$$|X_n|=\frac{1}{(x_1-y_1)} \prod_{i,j=2}^n\frac{(y_1-y_j)(x_1-x_i)}{(x_i-y_1)(x_1-y_i)}
	\begin{vmatrix}
	\frac{1}{(x_2-y_2)} & \frac{1}{(x_2-y_3)} & \cdots & \frac{1}{(x_2-y_n)} \\
	\frac{1}{(x_3-y_2)} & \frac{1}{(x_3-y_3)} & \cdots & \frac{1}{(x_3-y_n)} \\
	\vdots  &  \cdots  &  \vdots  & \vdots\\
	\frac{1}{(x_n-y_2)} & \frac{1}{(x_n-y_3)} &  \cdots  &  \frac{1}{(x_n-y_n)}     
	\end{vmatrix}.$$
	Repeating this procedure, we obtain that
	$$|X_n|=\frac{1}{\prod_{i\in I} (x_i-y_i)}\cdot 
	\frac{\prod_{i<j\in I} (y_i-y_j)(x_j-x_i) }{\prod_{i<j\in I} (x_i-y_j)(x_j-y_i) }
	=\frac{\prod_{i<j\in I} (x_i-x_j)(y_i-y_j)}{\prod_{i \in I} \prod_{j \in I} (x_i-y_j)}.$$
\end{proof}

In \cite{hilb1}, Hilbert introduced  a  certain square matrix which is a special case of the Cauchy square matrix. 
The {\it Hilbert's matrix}  is an  $n \times n$  matrix   $\Hb_n=[h_{ij}]$
with entries $h_{ij}=1/(i+j-1), $ where $ 1\leq i, j \leq n.$
Using the proposition \ref{Cauchy}, one can calculate the determinant of a  Hilbert's matrix as 
$$|\Hb_n|=\frac{c_n^4}{c_{2n}}, \ c_n = \prod_{i=1}^{n-1} i!.$$
He  also mentioned  that the determinant of $\Hb_n$ is the reciprocal of a well known  integer which follows from the following identity
$$\frac{1}{|\Hb_n|}= \frac{c_{2n}}{c_n^4} = n! \cdot \prod_{i}^{2n-1} \binom{i}{[i/2]}. $$
For more information see the sequence A005249 in  OEIS \cite{solane}. 
For a recent work related to the Cauchy's and Hilbert's matrices one can see \cite{mirf}.

The other type of matrices, which we are going to recall here, are the Toeplitz matrices.
An $n \times n$  {\it Toeplitz matrix} with entries in  a field $F$ is the square matrix 
\[V_n :=
\begin{bmatrix}
v_0 & v_1  & v_2 &\cdots & v_{n-1} \\
v_{-1} & v_0 & v_1 & \cdots & v_{n-2}\\
v_{-2} & v_{-1} & v_0 & \cdots & v_{n-3}\\
\vdots  &  \vdots &   \vdots  &  \cdots    & \vdots\\
v_{1-n} & v_{2-n} & v_{3-n} & \cdots & v_0\\    
\end{bmatrix}.\]

These are one of the most well studied and understood classes of  matrices that arise in most areas of the mathematics: algebra \cite{riets}, 
algebraic geometry \cite{englis}, and graph theory \cite{euler}.
In \cite{Chu-li}, the author   obtained a unique LU factorizations  and   an explicit formula for the determinant and also the inversion of Toeplitz matrices. And, the inverse, determinants, eigenvalues, and eigenvectors of  symmetric Toeplitz matrices over real number field  with linearly increasing entries have been studied in \cite{bunger}.
In \cite{ye-lim}, the author showed that every $n \times n$ square  matrix is generically a product of 
$\left\lfloor n/2 \right\rfloor +1$  and always a product of at most $2n + 5$ Toeplitz matrices.

\section{Determinant of certain square matrix}
\label{DCSM}
In this section, we  calculate the determinant of certain square matrices with entries in a field $F$ of characteristic zero, which are related to the determinant of  Cauchy's matrix. In special case, the determinant of our matrix is related to the determinant of a certain Toeplitz matrix. First, let us to give the following elementary result  for a given  infinite sequence
$\{a_\ell\}_{\ell=1}^\infty $ of distinct nonzero elements in a field $F$  of characteristic zero. 
\begin{lemma}
	\label{matl1} 
	For indexes $ e, \ell, s, $ and $ t$, we have 
	$$a_sd_{(\ell,e)}- a_\ell d_{(s,e)}=-a_ed_{(s,\ell)},\ 
	\frac{d_{(t,e)}}{d_{(t,\ell)}} -\frac{d_{(s,e)}}{d_{(s,\ell)}}=\frac{d_{(t,s)}d_{(\ell,e)}}{d_{(t,\ell)}d_{(s,\ell)}}.$$
\end{lemma}
%
\begin{proof}
	For indexes $ e, \ell,$ and $ s$, by definition  $d_{(s,e)}=d_{(s,\ell)}+d_{(\ell,e)}$, so 
	\begin{align*}
	a_s d_{(\ell ,e)}- a_\ell d_{(s,e)} & = a_s d_{(\ell,e)}- a_\ell (d_{(s,\ell )}+d_{(\ell ,e)}) \\
	& = (a_s-a_\ell ) d_{(\ell , e)} -a_\ell d_{(s, \ell )} \\ 
	& =  d_{(s,\ell )}(d_{(\ell , e)} -a_\ell ) = -a_e d_{(s, \ell )}
	\end{align*}
	For indexes $ e, \ell, s,$ and $ t$, one has 
	\begin{align*}
	\frac{d_{(t ,e)}}{d_{(t,\ell)}} -\frac{d_{(s, e)}}{d_{(s,\ell)}} & = 
	\frac{d_{(t, e)}d_{(s, \ell)}- d_{(s, e)}d_{(t, \ell)} }{d_{(t,\ell)} d_{(s, \ell)}}\\
	& = \frac{1}{d_{(t, \ell)} d_{(s, \ell)}}\cdot
	\begin{vmatrix}
	d_{(t, t)} & d_{(t, \ell)} \\
	d_{(s, e)} &  d_{(s, \ell)}
	\end{vmatrix}
	\\ 
	&   = \frac{1}{d_{(t, \ell)} d_{(s, \ell)}}\cdot
	\begin{vmatrix}
	d_{(t, e)}- d_{(s, e)}  & d_{(t, \ell)}-d_{(s, \ell)} \\
	d_{(s, e)} &  d_{(s, \ell)}
	\end{vmatrix}\\
	&   = \frac{1}{d_{(t, \ell)} d_{(s, \ell)}}\cdot
	\begin{vmatrix}
	d_{(t, s)} & d_{(t , s)} \\
	d_{(s, e)} &  d_{(s, \ell)}
	\end{vmatrix}\\
	&   = \frac{d_{(t, s)}}{d_{(t, \ell)} d_{(s, \ell)}}\cdot
	\begin{vmatrix}
	1 & 1\\
	d_{(s, e)} &  d_{(s, \ell)}
	\end{vmatrix}\\
	&   = \frac{d_{(t, s)}( d_{(s, \ell)}-  d_{(s, e)})}{d_{(t, \ell)} d_{(s, \ell)}}= 
	\frac{d_{(t, s)}d_{(\ell, e)}}{d_{(t, \ell)} d_{(s, \ell)}}.
	\end{align*}
\end{proof}


For any integer  $n \geq 1$, define  $(n+1) \times (n+1)$  matrix  $A_n$  as:
%
\[A_n := \begin{bmatrix}
1 & \frac{a_{i_1}}{d_{(i_1,e_1)}} & \frac{a_{i_1}}{d_{(i_1,e_2)}} &\cdots  &\frac{a_{i_1}}{d_{(i_1,e_n)}}  \\
1 &  \frac{a_{i_2}}{d_{(i_2,e_1)}} & \frac{a_{i_2}}{d_{(i_2,e_2)}} & \cdots & \frac{a_{i_2}}{d_{(i_2,e_n)}} \\
\vdots &     \vdots  &  \cdots  &  \vdots  & \vdots\\
1 &  \frac{a_{i_n}}{d_{(i_n,e_1)}} & \frac{a_{i_n}}{d_{(i_n,e_2)}} & \cdots  & \frac{a_{i_n}}{d_{(i_n,e_n)}}         \\
1 & \frac{a_{i_{n+1}}}{d_{(i_{n+1}, e_1)}} & \frac{a_{i_{n+1}}}{d_{(i_{n+1},e_2)}} &  \cdots  &  \frac{a_{i_{n+1}}}{d_{(i_{n+1},e_n)}} 
\end{bmatrix}, \]
where  $\{a_{i_1}, \cdots, a_{i_{n+1}}\}$ and $\{a_{e_1}, \cdots, a_{e_n}\}$ are disjoint subsets of the infinite sequence
$\{a_\ell\}_{\ell=1}^\infty$.
The following proposition  gives the determinant of $A_n$. 
We  will  use Lemma \ref{matl1}  in its proof.
\begin{proposition}
	\label{matp1}
	Let $I=\{1, 2, \cdots, n\} $ and $J=\{1, 2, \cdots, n+1\}$. Then, one has
	$$|A_n|= \frac{D_r \cdot \prod_{s' < s \in J}  d_{(i_s, i_{s'})}  }
	{\prod_{s\in J} \prod_{j \in I} d_{(i_s, e_j)} }.$$
\end{proposition}
\begin{proof}
	Subtracting first row from others and  using  lemma (\ref{matp1}),   gives that
	\begin{align*}
	|A_n|  &= \begin{vmatrix}
	1 & \frac{a_{i_1}}{d_{(i_1, e_1)}} & \cdots  &\frac{a_{i_1}}{d_{(i_1, e_n)}}  \\
	0 &  \frac{a_{i_2}d_{(i_1, e_1)}- a_{i_1}d_{(i_2, e_1)}}{d_{(i_1, e_1)}d_{(i_2, e_1)}}  & 
	\cdots & \frac{a_{i_2}d_{(i_1, e_n)} - a_{i_1}d_{(i_2, e_n)}}{d_{(i_1, e_n)}d_{(i_2, e_n)}} \\
	\vdots &     \vdots  &  \cdots  &  \vdots  \\
	0 & \frac{a_{i_{n+1}d_{(i_1, e_1)} - a_{i_1}} d_{(i_{n+1}, e_1)}}{d_{(i_1, e_1)}d_{(i_{n+1}, e_1)}} &  
	\cdots  &  \frac{a_{i_{n+1}}d_{(i_1, e_n)} - a_{i_1} d_{(i_{n+1}, e_n)}}{d_{(i_1, e_n)}d_{(i_{n+1}, e_n)}} 
	\end{vmatrix} \\
	& = \begin{vmatrix}
	1 & \frac{-a_{i_1}}{d_{(i_1, e_1)}} & \cdots  &\frac{-a_{i_1}}{d_{(i_1, e_n)}}  \\
	0 &  \frac{ -a_{e_1}d_{(i_2, i_1)}}{d_{(i_1, e_1)}d_{(i_2, e_1)}}  & 
	\cdots & \frac{ -a_{e_n}d_{(i_2, i_n)}}{d_{(i_1, e_n)}d_{(i_2, e_n)}} \\
	\vdots &     \vdots  &  \cdots  &  \vdots  \\
	0 & \frac{ -a_{e_1} d_{(i_{n+1}, i_1)}}{d_{(i_1, e_1)}d_{(i_{n+1}, e_1)}} &  
	\cdots  &  \frac{- a_{e_n} d_{(i_{n+1}, i_n)}}{d_{(i_1, e_n)}d_{(i_{n+1}, e_n)}} 
	\end{vmatrix}
	\end{align*}
	By extracting the factor $-a_{e_j}/d_{(i_1, e_j)}$ from each columns  $(1\leq j \leq n)$ and 
	$d_{(i_s, i_1)}$ from each rows $(2\leq s \leq n+1)$, one  gets  that
	$$|A_n| =(-1)^n\prod_{j=1}^n \frac{a_{e_j}  }{ d_{(i_1, e_j)}} \cdot \prod_{s=2}^{n+1} d_{(i_s, i_1)}\cdot |B_n|$$
	where
	\[B_n :=\begin{bmatrix}
	\frac{1}{d_{(i_2,e_1)}} & \frac{1}{d_{(i_2,e_2)}} & \cdots & \frac{1}{d_{(i_2,e_n)}} \\
	\frac{1}{d_{(i_3,e_1)}} & \frac{1}{d_{(i_3,e_2)}} & \cdots & \frac{1}{d_{(i_3,e_n)}} \\
	\vdots  &  \cdots  &  \vdots  & \vdots\\
	\frac{1}{d_{(i_{n+1}, e_1)}} & \frac{1}{d_{(i_{n+1},e_2)}} &  \cdots  &  \frac{1}{d_{(i_{n+1},e_n)}}.
	\end{bmatrix} \]  
	Since the matrix $B_n$ is  a  Cauchy's matrix defined by 
	$$x_1=a_{i_2}, \cdots, x_n=a_{i_{n+1}}, \ y_1=a_{e_1}, \cdots, y_n=a_{e_n},$$
	so using Proposition \ref{Cauchy}  we have
	$$|B_n| = 	
	\frac{ \prod_{s'<s \in \{2, \cdots, n+1\}}d_{(i_s, i_{s'})}  \cdot \prod_{i<j \in J} d_{( e_j, e_i)} }
	{\prod_{s=2}^{n+1}\prod_{j =1}^n d_{(i_s, e_j)} }, $$
	and hence,
	$$|A_n|=	\frac{(-1)^n\prod_{j=1}^n \frac{a_{e_j}  }{ d_{(i_1, e_j)}} \cdot \prod_{s'<s \in J } d_{(i_s, i_{s'})}}
	{\prod_{s=1}^{n+1}\prod_{j =1}^n d_{(i_s, e_j)}} =
	\frac{D_r \cdot \prod_{s'<s \in J } d_{(i_s, i_{s'})}}
	{\prod_{s=1}^{n+1}\prod_{j =1}^n d_{(i_s, e_j)}}.
	$$
\end{proof}
We note that the matrix $B_n$ in the proof of the above proposition is related to a certain $n \times n$ Toeplitz matrix. Indeed, if 
we consider the sequence   $a_\ell=1/\ell $ for $  \ell =1, 2, \cdots$ and indexes 
$e_j=j$ and $i_s=n+s-1 $ for $j=1, \cdots, n$ and $s=1, \cdots, n+1$, then a simple calculation shows that 
$B_n=(-1)^n (2n)! V_n,$ 
where $V_n$ is the following $n \times n$ Toeplitz matrix
$$V_n=
\begin{bmatrix}
\frac{1}{n} & \frac{1}{n-1}  & \frac{1}{n-2} &\cdots & \frac{1}{2}& 1 \\
\frac{1}{n+1} & \frac{1}{n} & \frac{1}{n-1} & \cdots & \frac{1}{3}& \frac{1}{2}\\
\frac{1}{n+2} & \frac{1}{n+1} & \frac{1}{n} & \cdots & \frac{1}{4}& \frac{1}{3}\\
\vdots  &  \vdots &   \vdots  &  \cdots    & \vdots  & \vdots\\
\frac{1}{2n-2} & \frac{1}{2n-3} &\frac{1}{2n-4} & \cdots & \frac{1}{n}& \frac{1}{n-1}\\
\frac{1}{2n-1} & \frac{1}{2n-2} &\frac{1}{2n-3} & \cdots & \frac{1}{n+1}& \frac{1}{n}\\    
\end{bmatrix}= (-1)^k \Hb_n, $$
where $k=n/2$ if $n$ is even and $k=(n-1)/2$ if $n$ is odd; and  the last equality comes by 
changing $j$-th column with $(n-j+1)$-th column of $V_n$.
\section{Proof of theorem \ref{matmain}}
In order to  prove Theorem  \ref{matmain},  we need the following result.
\begin{proposition}
	\label{matp3} 
	Given integers $2 \leq r< n$ satisfying  $r+1 \leq  n-r$, 
	consider the indexes  $r+1\leq i_1,  \cdots , i_{r+1} \leq n $ and  let $I=\{1, 2, \cdots, r\} $ and $J=\{1,  \cdots, r+1\}$.
	Then,
	the matrix $C':=[C_{i_s-r, j}]$, where $  s\in J$ and $0 \leq j \leq r$, has nonzero determinant as
	$$|C'|=(-1)^{r^2+3r} D_r^{r+1} \prod_{s'< s \in J} d_{(i_s, i_{s'})}.$$
\end{proposition}
\begin{proof}
	Extracting the factor  $(-1)^r \prod_{e< \ell \in I} d_{(\ell,e)}$  and  $(-1)^{r+j} \prod_{e< \ell \in I_j} a_\ell d_{(\ell,e)}$, respectively, from  first  column and the $j$-th column  for $s\in J$ and  $2 \leq j \leq r$, where $I_j=I \backslash \{j\}$,  gives that
	\begin{align*}
	|C'| & =  (-1)^{3r(r+1)/2} \prod_{e< \ell \in I} d_{(\ell,e)}  \prod_{e< \ell \in I_1} a_\ell d_{(\ell,e)}
	\prod_{e< \ell \in I_r} a_\ell d_{(\ell,e)}\cdot  |C''|\\
	&=  (-1)^{r(r+3)/2} D_r^{r} \cdot |C''|, 
	\end{align*}
	where $C''$ is  the following $(r+1)\times (r+1)$ matrix
	$$C'':=
	\begin{bmatrix}
	\prod_{\ell \in I} d_{(i_1,\ell)} &  a_{i_1} \prod_{\ell \in I_1}  d_{(i_1,\ell)} 
	& \cdots &  a_{i_1} \prod_{\ell \in I_r}d_{(i_1,\ell)} \\
	\prod_{\ell \in I} d_{(i_2,\ell)} &  a_{i_2} \prod_{\ell \in I_1}  d_{(i_2,\ell)}& 
	\cdots & a_{i_2} \prod_{\ell \in I_r} d_{(i_2,\ell)} \\
	\vdots  &  \cdots  &  \vdots  & \vdots\\
	\prod_{\ell \in I} d_{(i_{r+1},\ell)} &  a_{i_{r+1}} \prod_{\ell \in I_1}  d_{(i_{r+1},\ell)} & 
	\cdots  & a_{i_{r+1}} \prod_{\ell \in I_r} d_{(i_{r+1},\ell)}
	\end{bmatrix}.$$
	By extracting the factor $\prod_{\ell \in I} d_{(i_s,\ell)}$ from $s$-th row  $1 \leq s \leq r+1$, we  obtain  
	$$|C''|=\prod_{s\in J} \prod_{j \in I} d_{(i_s,j)} \cdot 
	\begin{vmatrix}
	1 & \frac{a_{i_1}}{d_{(i_1,1)}} & \cdots  &\frac{a_{i_1}}{d_{(i_1,r)}}  \\
	1 &  \frac{a_{i_2}}{d_{(i_2,1)}} & \cdots & \frac{a_{i_2}}{d_{(i_2,r)}} \\
	\vdots &     \vdots  &  \cdots    & \vdots\\
	1 & \frac{a_{i_{r+1}}}{d_{(i_{r+1}, 1)}} &   \cdots  &  \frac{a_{i_{r+1}}}{d_{(i_{r+1},r)}} 
	\end{vmatrix}$$
	Considering $t=r$ and $e_j=j$ for $j=1, \cdots, r$, and using Proposition  \ref{matp1} for calculating  the last determinant, one can  conclude  that
	\begin{align*}
	|C'| & =(-1)^{r(r+3)/2 } D_r^{r-1} \prod_{s\in J} \prod_{j \in I} d_{(i_s,\ell)} \cdot 
	\frac{D_r \cdot \prod_{s' < s \in J}  d_{(i_s, i_{s'})}  } {\prod_{s\in J} \prod_{j \in I} d_{(i_s, j)} }\\
	& = (-1)^{r^2+3r} D_r^{r+1} \prod_{s'< s \in J} d_{(i_s, i_{s'})}.
	\end{align*}
	
\end{proof}
We notice that above proposition  is a special case of the next general one.
\begin{proposition}
	\label{matp4} 
	Given integers $2 \leq r< n$, and $m\leq \min \{ r+1, n-r\}$, any $m\times m$ sub-matrix of 
	the matrix $C=[C_{(i-r,j)}]$ has non-zero determinant,  
	where $r+1 \leq i \leq n$ and  $0\leq j \leq r$; therefore
	$C$ has maximal rank equal to $\min \{n-r, r+1\}$.  
\end{proposition}
\begin{proof}
	We may assume that $r+1\leq n-r$, the other case is similar. 
	For $ m \leq r+1$,  we denote by  $C_m$ any  $m\times m$ sub-matrix of $C$. 
	By proposition \ref{matp3}, the determinant of $C_m$ is nonzero for $m=r+1$.  
	Thus,  we may suppose that $m< r+1$ and $C_m=[C_{(i_s-r , j_t)}]$, where 
	$r+1\leq i_s \leq n-r$, $0 \leq j_{s'} \leq r$ for $1\leq s, s' \leq m$.
	If we suppose that $0=j_1 < j_2 , \cdots , j_m$, then 
	$$C_m= \begin{bmatrix}
	C_{(i_1-r, 0)}& C_{(i_1-r, j_2)}&  \cdots & C_{(i_1-r, j_m)}\\
	C_{(i_2-r, 0)}&C_{(i_2-r, j_2)}& \cdots & C_{(i_2-r, j_m)}\\
	\vdots &     \vdots  &  \cdots    & \vdots\\
	C_{(i_m-r, 0)}&C_{(i_m-r, j_2)}& \cdots & C_{(i_m-r, j_m)}.
	\end{bmatrix},$$
	such that 
	$$ C_{(i_s-r,0)} =(-1)^{r}  \prod_{ e< \ell  \in I}  d_{(i_s,\ell)} d_{(\ell,e)},$$
	$$C_{(i_s-r,j_{s'})} =  (-1)^{r+j_{s'}} a_{i_s}   \prod_{ e< \ell  \in I_{j_{s'}}} a_\ell d_{(i_s,\ell)} d_{(\ell,e)},$$
	where $I=\{1,2, \cdots, r\}$ and $I_{j_{s'}}= I \backslash \{j_{s'}l \} $.
	Extracting  $(-1)^{r}  \prod_{ e< \ell  \in I}  d_{(\ell,e)}$ and
	$(-1)^{r+j_{s'}}  \prod_{ e< \ell  \in I_{j_{s'}}} a_\ell d_{(\ell,e)}$ from first and $s'$-th columns, respectively, and then 
	$\prod_{\ell  \in I}  d_{(i_s, \ell)}$ from $s$-th row for  $1 \leq s < s'\leq m$, gives that
	\begin{align*}
	|C_m| &= (-1)^{r'}  \prod_{ e< \ell  \in I}  d_{(\ell,e)} \cdot  \prod_{ t=2}^m \prod_{ e< \ell  \in I_{j_t}} a_\ell d_{(\ell,e)}
	\cdot \prod_{ s=1}^m \prod_{ \ell  \in I} d_{(i_s, \ell)} \\
	& \times  
	\begin{vmatrix}
	1 & \frac{a_{i_1}}{d_{(i_1,j_2)}} &  \cdots  &\frac{a_{i_1}}{d_{(i_1,j_m)}}  \\
	1 &  \frac{a_{i_2}}{d_{(i_2,j_2)}} &  \cdots & \frac{a_{i_2}}{d_{(i_2,j_m)}} \\
	\vdots &     \vdots  &  \cdots  &  \vdots  \\
	1 &  \frac{a_{i_m}}{d_{(i_m,j_2)}} & \cdots  & \frac{a_{i_m}}{d_{(i_m,j_m)}}        
	\end{vmatrix},
	\end{align*}
	where $r'= mr + j_2 + \cdots + j_m$  and the above is nonzer by Propositions \ref{matp1}.
	Otherwise, if suppose that $1 \leq j_1 < j_2 < \cdots < j_m$, then extracting the factor 
	$(-1)^{r+j_{s'}} a_{i_s}   \prod_{ e< \ell  \in I_{j_{s'}}} a_\ell d_{(\ell,e)}$ from ${s'}$-th column, and then 
	$\prod_{\ell  \in I}  d_{(i_s, \ell)}$ from $s$-th row  of the matrix $C_m=[C_{i_s-r, j_{s'}}]$, where  $1 \leq s, s' \leq m$, 
	gives that 
	\begin{align*}
	|C_m| &= (-1)^{r''}  \prod_{s, t=1}^m  \prod_{ e< \ell  \in I_{j_t}} a_\ell d_{(\ell,e)}  d_{(i_s, \ell)} \\
	& \times  
	\begin{vmatrix}
	\frac{1}{d_{(i_1,j_1)}} & \frac{1}{d_{(i_1,j_2)}} & \cdots & \frac{1}{d_{(i_1,j_m)}} \\
	\frac{1}{d_{(i_2,j_1)}} & \frac{1}{d_{(i_2,j_2)}} & \cdots & \frac{1}{d_{(i_2,j_m)}} \\
	\vdots  &  \cdots  &  \vdots  & \vdots\\
	\frac{1}{d_{(i_m, j_1)}} & \frac{1}{d_{(i_m,j_2)}} &  \cdots  &  \frac{1}{d_{(i_m,j_m)}}     
	\end{vmatrix},
	\end{align*}
	where $r''= mr + j_1 + \cdots + j_m$ and the last determinant is nonzero
	by Propositions \ref{Cauchy}. This completes the proof of the proposition.
\end{proof}
Now we are ready to prove the main theorem  \ref{matmain}, using the above results.
\begin{proof}
	For integers $2 \leq r < n$, recall that ${\bf C}_r^n:=[C|D]$ is a $(n-r)\times(n+1)$   blocked matrix,
	where $C=[C_{(i-r,j)}]$ is $(n-r)\times(r+1)$ matrix  defined as in the first section and   $D$ is a $(n-r)\times(n-r)$ 
	diagonal matrix with entries $D_r$. By Proposition \ref{matp4},  any $m\times m$ sub-matrix of 
	the matrix $C$ has non-zero determinant and  $C$ has maximal rank equal to $\min \{n-r, r+1\}$.  It is clear that the matrix $D$ has full rank equal to $n-r$. By exchanging the columns, if it is necessary,  one can see that any  $(n-r)\times(n-r)$ submatrix of ${\bf C}_r^n$  is a diagonal blocked matrix with blocks equal to $D_r$ or $m \times m$ submatrices of $C$ with   $1 \leq m \leq \min \{n-r, r+1\}$, which have non-zero determinant. Therefore, any  $(n-r)\times(n-r)$ submatrix of ${\bf C}_r^n$  has nonzero determinant, and hence it has maximal rank $n-r$, as desired.
\end{proof}


\end{document}